\documentclass{article}

\usepackage{amssymb,amsmath,amsthm,mathrsfs}
\usepackage{verbatim}
\usepackage[a4paper,includeheadfoot,left=30mm,right=30mm,top=25mm,bottom=30mm]{geometry} 
\usepackage{graphicx}
\usepackage[round]{natbib}

\newtheorem{thm}{Theorem}[section]

\newtheorem{prop}[thm]{Proposition}

\newtheorem{defn}[thm]{Definition}

\theoremstyle{definition}
\newtheorem{cond}[thm]{Condition}

\newtheorem{rmk}[thm]{Remark}
\theoremstyle{remark}

\newcommand{\TextNorm}[1]{\textrm{\textmd{\textup{#1}}}}
\newcommand{\hsforall}{\hspace{1mm}\forall\hspace{1mm}}						 
\newcommand{\res}{\operatorname*{Res}}                             
\renewcommand{\Re}{\operatorname*{Re}}                             
\renewcommand{\Im}{\operatorname*{Im}}                             
\renewcommand{\d}{\ensuremath{\,\mathrm{d}}}							         
\newcommand{\T}{\text{\TextNorm{T}}\hspace{0.1mm}}							   
\newcommand{\DeltaP}{\Delta_{\mathrm{PDE}}\hspace{0.5mm}}          
\newcommand{\DeltaPsup}[1]{\Delta_{\mathrm{PDE}}^{#1}\hspace{0.5mm}}
\newcommand{\Mspacer}{\hspace{0.5mm}}                              
\newcommand{\M}[3]{#1_{#2\Mspacer#3}}                              
\newcommand{\Msup}[4]{#1_{#2\Mspacer#3}^{#4}}                      
\newcommand{\Msups}[5]{#1_{#2\Mspacer#3}^{#4\Mspacer#5}}           
\renewcommand{\geq}{\geqslant}                                     
\newcommand{\BE}{\begin{equation}}                                 
\newcommand{\EE}{\end{equation}}                                   
\newcommand{\BES}{\begin{equation*}}                               
\newcommand{\EES}{\end{equation*}}                                 
\newcommand{\BP}{\begin{pmatrix}}                                  
\newcommand{\EP}{\end{pmatrix}}                                    
\newcommand{\N}{\mathbb{N}}                                        
\newcommand{\R}{\mathbb{R}}                                        
\newcommand{\C}{\mathbb{C}}                                        

\newcommand{\superscript}[1]{\ensuremath{^{\textrm{#1}}}}

\newcommand{\Thns}[0]{\superscript{th}}
\newcommand{\Th}[0]{\Thns~}

\newcommand{\rdns}[0]{\superscript{rd}}
\newcommand{\rd}[0]{\rdns~}

\def\clap#1{\hbox to 0pt{\hss#1\hss}}

\numberwithin{equation}{section}

\begin{document}

\title{Spectral theory of some non-selfadjoint \\ linear differential operators}
\author{B. Pelloni$^1$ and D. A. Smith$^2$ \\
\footnotesize $^1$ Department of Mathematics, University of Reading RG\textup{6} \textup{6}AX, UK \\
\footnotesize $^2$ \emph{Corresponding author}, ACMAC, University of Crete, Heraklion 71003, Crete, Greece \\
\footnotesize email\textup{: \texttt{d.a.smith@acmac.uoc.gr}}
}
\date{\today}

\maketitle

\abstract{
We give a characterisation of the spectral properties of linear differential operators with constant coefficients, acting on functions defined on a bounded interval, and determined  by general linear  boundary conditions. The boundary conditions may be such that the resulting operator is not selfadjoint.

We associate the spectral properties of such an operator $S$ with the properties of the solution of a corresponding boundary value problem for the partial differential equation $\partial_t q \pm iSq=0$. Namely, we are able to establish an explicit correspondence between the properties of the family of eigenfunctions of the operator, and in particular whether this family is a basis, and the existence and properties of the unique solution of the associated boundary value problem. When such a unique solution exists, we consider its representation as a complex contour integral that is obtained using a transform method recently proposed by Fokas and one of the authors. The analyticity properties of the integrand in this representation are crucial for studying the spectral theory of the associated operator.
}


\section{Introduction}

In this paper, we study the following two objects: 
\begin{itemize}
\item[(1)]
A linear constant-coefficient differential operator $S$ defined on a domain of the form
$\mathcal{D}(S)=\{u\in L^2[0,1]: u$ sufficiently smooth and satisfying $n$ prescribed boundary conditions$\}$.
\item[(2)]
An initial boundary value problem (IBVP) for the linear evolution partial differential equation $q_t(x,t)\pm iSq(x,t)=0$, $x\in(0,1)$ $t\in(0,T)$, with $S$ as in (1), an initial condition $q(x,0)=q_0(x)$ and $n$ given boundary conditions. 
\end{itemize}
The boundary conditions, assumed to be linear, can be prescribed at either end of the interval $[0,1]$, or can couple the two ends.

It is to be expected that the objects (1) and (2) are closely related. 
For each of these objects, it is natural to formulate a basic question, whose answer depends on the specific boundary conditions. Namely, given a set of $n$ boundary conditions,
\begin{itemize}
\item[(Q1)]
does the resulting operator $S$ admit a basis of eigenfunctions, in any appropriate sense?
\item[(Q2)]
does the resulting initial-boundary value problem admit a unique solution representable by a discrete series expansion in the eigenfunctions of $S$?
\end{itemize}

Although it should be clear that these are the same question posed in different contexts, very little is explicitly known beyond the classical cases when the spatial operator has a known basis of eigenfunctions. This basis can be used after separation of variables to express the solution of the boundary value problem. 

\smallskip

In this paper we give an explicit connection between the two problems in general; we give a link between the solutions of (1) and (2), and we show precisely how the answer to (Q1) and (Q2) are related. In particular, the rigorous answer to one question can be given through answering the other. Our results are true for general $n$, however they are new and interesting in particular for $n$ odd.

Since in general $S$ will not be self-adjoint, we expect that any spectral decomposition involves not only $S$ but also the adjoint $S^*$.  
In terms of the PDE problem, we will see that this is reflected in the need to consider both the initial time and the final time problems (the evolution with reversed time direction).

\subsubsection*{The operator problem}

We consider the linear ordinary differential operator $S$, given by
\BE \label{eqn:defn.S}
Su = \left(-i\frac{\d}{\d x}\right)^n u, \qquad u\in\mathcal{D}(S),
\EE
defined on the domain $\mathcal{D}(S)\subset L^2[0,1]$ given by
\BE
\mathcal{D}(S) = \{u\in AC^n[0,1]: A(u^{(n-1)}(0),u^{(n-1)}(1),\dots,u(0),u(1))=0\},
\EE
where
\BE
AC^n=\{f\in C^{n-1}: f^{(n-1)} \mbox{ absolutely continuous and } f^{(n)}\in L^2[0,1]\}.
\EE
By $\overline{\mathcal{D}}(S)$, we denote the $L^2$ closure of $\mathcal{D}(S)$. Here the \emph{order} $n\geq2$ is an integer and the \emph{boundary coefficient matrix} $A\in\mathbb{R}^{n\times2n}$, encoding the given boundary conditions, is of rank $n$ and given, in reduced row-echelon form, by 
\BE \label{eqn:Boundary.Coefficients}
A=\BP \M{\alpha}{1}{n-1} & \M{\beta}{1}{n-1} & \M{\alpha}{1}{n-2} & \M{\beta}{1}{n-2} & \dots & \M{\alpha}{1}{0} & \M{\beta}{1}{0} \\
\M{\alpha}{2}{n-1} & \M{\beta}{2}{n-1} & \M{\alpha}{2}{n-2} & \M{\beta}{2}{n-2} & \dots & \M{\alpha}{2}{0} & \M{\beta}{2}{0} \\ 
\vdots & \vdots & \vdots & \vdots & & \vdots & \vdots \\
\M{\alpha}{n}{n-1} & \M{\beta}{n}{n-1} & \M{\alpha}{n}{n-2} & \M{\beta}{n}{n-2} & \dots & \M{\alpha}{n}{0} & \M{\beta}{n}{0} \EP.
\EE
The numbers $\M{\alpha}{j}{r}$, $\M{\beta}{j}{r}$ are called the \emph{boundary coefficients}. 

\smallskip
This operator has been studied at least since~\citet{Bir1908b}. Depending on the particular entries of the matrix $A$, the operator may or may not be selfadjoint. The theory of the selfadjoint case was fully understood by the time~\citet*{DS1963a} presented it.

\citet{Loc2000a,Loc2008a} used the theory of Fredholm operators to study the non-selfadjoint case. He defined the \emph{characteristic determinant}
\BE
\Delta(\rho) = e^{i\sum_{k=1}^{\nu-1}\omega^k\rho}\det M(\rho), 
\label{chardet}
\EE
where $\omega=\exp(2\pi i/n)$ and the entries of the matrix $M(\rho)$ are given by
$$
\M{M}{k}{j} (\rho)= \sum_{r=1}^n \M{\alpha}{j}{r}(i\omega^{k-1}\rho)^{r}e^{-i\omega^{k-1}\rho} + \sum_{r=1}^n \M{\beta}{j}{r}(i\omega^{k-1}\rho)^{r}.
$$

It is known that, provided $\Delta\neq0$, if $\Delta(\sigma)=0$ then $\sigma^n$ is an eigenvalue of $S$. Further, the algebraic multiplicity of $\sigma^n$ as an eigenvalue of $S$ is equal to the order of $\sigma$ as a zero of $\Delta$. Locker showed that, for Birkhoff-regular operators, the generalised eigenfunctions form a complete system. However, he gives no general statement about the cases that do not satisfy these regularity conditions. 

\subsubsection*{The PDE problem}
In a separate development, a novel transform method for analysing IBVPs was developed by Fokas~\citep[see][for an overview]{Fok2008a}. The method was applied to IBVPs posed for evolution equations on the half-line by~\citet{FS1999a} and on the finite interval by~\citet{FP2001a} with simple, uncoupled boundary conditions. In~\citet{Smi2012a}, Fokas' method was applied to IBVPs whose spatial part is given by the operator $S$, namely those of the form
\BE \label{eqn:PDE}
\partial_t q(x,t) + a (-i\partial_x)^n q(x,t) = 0,\quad x\in(0,1), \quad t>0,\quad a=\pm i,
\EE
with prescribed boundary conditions and an initial condition $q_0(x)=q(x,0)$, assumed smooth to avoid technical complications. Usually the initial condition is assumed to be in $C^\infty$. However, the same results hold assuming that $q_0\in AC^n$. Indeed, in this case, the uniform convergence of the integral representation (see~\eqref{pdeintrepr} below), the poynomial decay rate of the integrand and the explicit exponential $x$ dependence imply that the solution $q$ belongs to the same class. In what follows we assume $q_0\in AC^n$.

This method yields an integral representation of the solution of the initial-boundary value problem in the form
\begin{multline}
q(x,t) = \frac 1 {2\pi} \int_{\Gamma^+}{\rm e}^{i\rho x-a\rho^nt}\frac{\zeta^+(\rho)}{\DeltaP}\d\rho + \frac 1 {2\pi} \int_{\Gamma^-}{\rm e}^{i\rho(x-1)-a\rho^nt}\frac{\zeta^-(\rho)}{\DeltaP}\d\rho \\
+ i \sum_{k\in K^+} {\rm e}^{i\sigma_k x-a\sigma_k^nt}\res_{\rho=\sigma_k}\frac{\zeta^+(\rho)}{\DeltaP}\d\rho + i \sum_{k\in K^-} {\rm e}^{i\sigma_k(x-1)-a\sigma_k^nt}\res_{\rho=\sigma_k}\frac{\zeta^-(\rho)}{\DeltaP}\d\rho,
\label{pdeintrepr}
\end{multline}
where the quantities $\hat q_0$, $\zeta^\pm$, $\DeltaP$, $\sigma_k$ and $\Gamma^\pm$ are defined below in Definitions~\ref{defn:PDE.Char.Det} and~\ref{defn:sigmak}. In many cases, including all problems with $n$ even, the integrals in equation~\eqref{pdeintrepr} both evaluate to zero~\citep{Smi2012a}. We study these cases here.

In~\citet{Pel2004a,Pel2005a} and then in greater generality in~\citet{Smi2011a}, this method is used to characterise boundary conditions that determine well-posed problems, and problems whose solutions admit representation by series. To achieve this characterisation, the central objects of interest are the {\em PDE characteristic matrix} $\mathcal{A}$ (see Definition~\ref{defn:PDE.Char.Det} below) and its determinant $\DeltaP$.

Note that in this work, by `well-posed', we mean existence and uniqueness of a solution and make no claim to continuity with respect to data. By `ill-posed' we mean that existence or uniqueness fails. The results of~\citet{FS1999a,Pel2004a,Smi2012a} establish that a problem is well-posed if and only if it admits a solution via the method of Fokas.

The present work details results connecting the spectral theory of $S$ with the behaviour of the associated IBVPs for the PDE  (\ref{eqn:PDE}), as well as the one obtained from the same set of boundary conditions but for the PDE
\BE \label{eqn:PDEft}
\partial_t q(x,t) - a (-i\partial_x)^n q(x,t) = 0,\quad x\in(0,1), \quad t>0.
\EE
We refer to the latter in the sequel as the {\em final time boundary value problem}.

\subsubsection*{Summary of the main results} 

For an operator $S$ of the type given by~\eqref{eqn:defn.S}, and the associated initial- and final-boundary value problems, we prove the following:

\begin{itemize}
\item
{\em If the eigenfunctions of $S$ and $S^*$ form a biorthogonal basis of $\overline{\mathcal{D}}(S)$ and the IBVP is well posed, then its solution is representable as a series}.

\smallskip
This is the content of Proposition~\ref{prop:basis.then.both.wellposed}. It follows from this result that if a series representation does not exist, then the eigenfunctions of $S$ and $S^*$ cannot form a basis of $\overline{\mathcal{D}}(S)$. What is interesting is that we can use the PDE approach to obtain results on $S$ in cases that are not covered by usual operator theoretic techniques.
In section~\ref{sec:Eg2} we provide an example when~(Q1) cannot be answered by the usual tests involving projector norms, but may be settled through this result and a negative answer to~(Q2). 

\item
{\em If the initial- and final-boundary value problems are well posed, then the eigenfunctions of $S$ and $S^*$ form a complete biorthogonal system in $\overline{\mathcal{D}}(S)$}.

\smallskip
This is the content of Theorem~\ref{thm:both.wellposed.then.complete}. The conclusion does not imply that the eigenfunctions necessarily form a basis.
However the integral representation (\ref{pdeintrepr}) can {\em always} be deformed to derive a {\em series representation} for the solution of the IBVP in terms of the eigenfunctions. 

\item
{\em The departure of the family of eigenfunctions of $S$ and $S^*$ from being a biorthogonal basis can be estimated in terms of the integrand in the representation of the solution of the associated IBVP}.

\smallskip
This is the content of Theorem~\ref{thm:Davies.Norms.2}. This departure is quantified in the notion of `wildness'~\citep[see][]{Dav2007a}.
Indeed, if the eigenfunction of $S$ and $S^*$ form a wild system in $L^2[0,1]$, then we provide an estimate of the wildness of the system in terms of the quantities used to determine whether the initial- and final-boundary value problems are well posed.
\end{itemize}

\subsubsection*{Outline of paper}

In section~\ref{sec:Theorems}, we review the necessary definitions and notation. Following this, we precisely state and prove the results described above.

Each of sections~\ref{sec:Eg1} and~\ref{sec:Eg2} is devoted to the analysis of an example which illustrates the above general results. We compare and contrast the results obtained through the new theorems with those yielded by Davies' wildness method.

\section{Complete and basic systems of eigenfunctions} \label{sec:Theorems}

\subsection{Notation, definitions and preliminary results}
In this paper, we make extensive use of the notation developed in~\citet{Smi2012a}. We refer to that paper for details, but we list here some of the notation used throughout the rest of this work.

\noindent{\bfseries The initial-boundary value problem $\Pi(n,A,a,q_0)$:}
Find $q \in AC^n([0,1]\times[0,T])$ which satisfies the linear, evolution, constant-coefficient partial differential equation
\BE \label{eqn:P1:Intro:PDE}
\partial_tq(x,t) + a(-i\partial_x)^nq(x,t) = 0
\EE
subject to the initial condition
\BE \label{eqn:P1:Intro:IC}
q(x,0) = q_0(x)
\EE
and the boundary conditions
\BE \label{eqn:P1:Intro:BC}
A\left(\partial_x^{n-1}q(0,t),\partial_x^{n-1}q(1,t),\partial_x^{n-2}q(0,t),\partial_x^{n-2}q(1,t),\dots,q(0,t),q(1,t)\right)^\T = h(t),
\EE

\noindent where the quadruple $(n,A,a,q_0)\in\mathbb{N}\times\mathbb{R}^{n\times2n}\times\mathbb{C} \times AC^n[0,1]$ is such that
\begin{description}
\item[$(\Pi1)$]{the \emph{order} $n\geq 2$,}
\item[$(\Pi2)$]{the \emph{boundary coefficient matrix} $A$ is in reduced row-echelon form,}
\item[$(\Pi3)$]{the \emph{direction coefficient} has the specific value $a=\pm i$,}
\item[$(\Pi4)$]{the \emph{initial datum} $q_0$ is compatible with the boundary conditions in the sense
\BE \label{eqn:P1:Intro:Compatibility}
A\left(q_0^{(n-1)}(0),q_0^{(n-1)}(1),q_0^{(n-2)}(0),q_0^{(n-2)}(1),\dots,q_0(0),q_0(1)\right)^\T = 0.
\EE}
\end{description}

Given a problem $\Pi=\Pi(n,A,a,q_0)$, we define the corresponding \emph{final time  time problem} $\Pi'=\Pi(n,A,-a,q_0)$.

\smallskip
We assume that the boundary conditions are homogeneous to aid the comparison with $S$, the differential operator representing the spatial part of the PDE problem $\Pi$. There is no loss of generality in this assumption. Without this restriction, $\Pi$ is no more difficult to solve; the solution simply contains an additional term represented as an integral along the real line~\citep{Smi2012a}.

\begin{defn} \label{defn:PDE.Char.Det}
Let $\Msup{\alpha}{k}{j}{\star}$, $\Msup{\beta}{k}{j}{\star}$ be the boundary coefficients of the operator $S^\star$, adjoint to $S$. We define
\begin{align}
\Msup{\mathcal{A}}{k}{j}{+}(\rho) &= \sum_{r=0}^{n-1}(-i\omega^{k-1}\rho)^r \Msup{\alpha}{k}{j}{\star}, \\
\Msup{\mathcal{A}}{k}{j}{-}(\rho) &= \sum_{r=0}^{n-1}(-i\omega^{k-1}\rho)^r \Msup{\beta}{k}{j}{\star}, \\
\mbox{then } \M{\mathcal{A}}{k}{j}(\rho) &= \Msup{\mathcal{A}}{k}{j}{+}(\rho) + \Msup{\mathcal{A}}{k}{j}{-}(\rho)e^{-i\omega^{k-1}\rho}
\end{align}
is called the \emph{PDE characteristic matrix}. The determinant $\DeltaP$ of $\mathcal{A}$ is called the \emph{PDE characteristic determinant}.
\end{defn}

\begin{rmk}
The PDE characteristic matrix is a realisation of Birkhoff's characteristic matrix for $S^\star$ and also represents the Dirichlet-to-Neumann map for the problem $\Pi$. Indeed, it is through this matrix that the unknown (Neumann) boundary values are obtained from the (Dirichlet) boundary data of the problem. \citet{Smi2012a} uses a different but equivalent definition of $\mathcal{A}$ which generalises the construction via determinants and Cramer's rule originally found in~\citet{FS1999a}. The validity of the new definition is established in~\citet{FS2013a} and the equivalence is explicitly proven in~\citet{Smi2013b}.
\end{rmk}

\begin{rmk}
In Definition~\ref{defn:PDE.Char.Det}, we construct $\mathcal{A}$ via the boundary conditions of $S^\star$. It is possible to make an alternative but equivalent definition of $\mathcal{A}$ via an explicit construction from the boundary conditions of $S$ itself. For the examples considered in sections~\ref{sec:Eg1}--\ref{sec:Eg2}, this is a simple matter. Indeed, provided the boundary conditions of $S$ are non-Robin,~\citet[Lemma~2.14]{Smi2011a} provides a simple construction. This can be done for general boundary conditions~\citep{Smi2012a} and can easily be coded to be done automatically.
\end{rmk}

\begin{defn}\label{defn:sigmak}
Let $(\sigma_k)_{k\in\N}$ be a sequence containing each nonzero zero of $\DeltaP$ precisely once. We define the index sets $K^+=\{k\in\N:\sigma_k\in\overline{\C^+}\}$, $K^- = \{k\in\N:\sigma_k\in\C^-\}$. Let $3\epsilon$ be the infimal separation of the zeros $\sigma_k$. Then the contours $\Gamma^\pm$ are the positively-oriented boundaries of
\BE
\{\rho\in\C^\pm:\Re(a\rho^n)>0\} \setminus \bigcup_{k\in\N}B(\sigma_k,\epsilon).
\EE

The minor $X^{r\hspace{0.5mm}j}(\rho)$ is the $(n-1)\times(n-1)$ submatrix of $\mathcal{A}$ whose (1,1) entry is the (r+1,j+1) entry. This is used to construct the spectral functions
\begin{align}\label{zeta+}
\zeta^+(\rho,q_0) = \sum_{r=1}^n\sum_{j=1}^n \det X^{r\hspace{0.5mm}j}(\rho) \Msup{\mathcal{A}}{1}{j}{+}(\rho) \hat{q}_0(\omega^{r-1}\rho), \\ \label{zeta-}
\zeta^-(\rho,q_0) = \sum_{r=1}^n\sum_{j=1}^n \det X^{r\hspace{0.5mm}j}(\rho) \Msup{\mathcal{A}}{1}{j}{-}(\rho) \hat{q}_0(\omega^{r-1}\rho),
\end{align}
where
$$
\hat q_0(\rho)=\int_0^1e^{-i\rho x}q_0(x)\d x.
$$
\end{defn}

\begin{defn}\label{defn:Conditioning}
We say the IBVP is \emph{well-conditioned} if it satisfies:

\emph{$\zeta^\pm(\rho)$ is entire and the ratio
\BE \label{eqn:P1:Intro:thm.WellPosed:Decay}
\frac{\zeta^\pm(\rho)}{\DeltaP(\rho)}\to0\qquad\begin{matrix}\mbox{ as }\rho\to\infty \mbox{ from within a sector exterior}\\\mbox{to } \Gamma^\pm, \mbox{ away from the zeros of } \DeltaP.\end{matrix}
\EE
}

Otherwise, we say that the problem is \emph{ill-conditioned}.
\end{defn}

Well-conditioning of an IBVP is not a classical definition and is unrelated to the concept of conditioning that appears in numerical analysis. Conditioning, in the sense of Definition~\ref{defn:Conditioning}, is necessary for well-posedness but is also central to the validity of a series representation. Indeed, switching the direction coefficient $a\mapsto -a$ in the PDE~\eqref{eqn:PDE} switches which sectors are enclosed by the contours $\Gamma^\pm$ thus, by Jordan's Lemma, well-conditioning of the problem with the opposite direction coefficient is equivalent to the two integrals in~\eqref{pdeintrepr} vanishing~\citep{Smi2012a}.

\smallskip

The reader will recall that a system $(\phi_n)_{n\in\N}$ in a Banach space is said to be \emph{complete} if its linear span is dense in the space and such a system is a \emph{basis} if for each $f$ in the space there exists a unique sequence of scalars $(\alpha_n)_{n\in\N}$ such that
\BES
f = \lim_{r\to\infty} \left(\sum_{n=1}^r\alpha_n\phi_n\right).
\EES

\subsection{Well-posed PDE systems and bases of eigenfunctions}

It is well known~\citep[see][Section~12.5]{CL1955a} that if the zeros of the characteristic determinant $\Delta$ of $S$ are all simple then the eigenfunctions of $S$ form a complete system in $\overline{\mathcal{D}}(S)$. This theorem is proven using an analysis of the Green's functions of both the operator $S$ and its adjoint $S^\star$. We prove the following result without directly analysing the adjoint operator.

\begin{thm} \label{thm:both.wellposed.then.complete}
Let $S$ be such that the zeros of $\DeltaP$ are all simple. Let $\Pi=\Pi(n,a,A,q_0,0)$ be an IBVP associated with $A$ and $\Pi'$ be the corresponding problem with the opposite direction coefficient, $\Pi(n,-a,A,q_0,0)$. If $\Pi$ is well-posed and $\Pi'$ is well-conditioned in the sense of Definition~\ref{defn:Conditioning} then the eigenfunctions of $S$ form a complete system in $\overline{\mathcal{D}}(S)$.
\end{thm}

Rather than analysing both the original operator $S$ and the adjoint operator $S^\star$, one needs information on both the initial- and final-boundary value problems associated with the operator $S$.

A stronger, but essentially straightforward, result in the reverse direction is:

\begin{prop} \label{prop:basis.then.both.wellposed}
If the eigenfunctions of $S$ form a basis in $\overline{\mathcal{D}}(S)$ and, for some $a$, the associated IBVP $\Pi$ is well-posed, then $\Pi'$ is well-conditioned.

Further, if $(\phi_k)_{k\in\N}$ are the eigenfunctions of $S$, with corresponding eigenvalues $(\sigma_k^n)_{k\in\N}$ then there exists a sequence $(\psi_k)_{k\in\N}$ biorthogonal to $(\phi_k)_{k\in\N}$ such that the Fourier expansion
\BE \label{eqn:basis.then.both.wellposed:q}
\sum_{k\in\N}\phi_k(x)\langle q_0,\psi_k \rangle e^{-\sigma_k^nt}
\EE
converges to the solution of $\Pi$.
\end{prop}

Indeed, in the notation of Proposition~\ref{prop:basis.then.both.wellposed}, each $\psi_k$ is an eigenfunction of the adjoint operator $S^\star$ with corresponding eigenvalue $-\sigma_k^n$~\citep{Bir1908a}.

The above results are essentially the translation into operator theory language of results proved in~\citet{Smi2011a}. Here we extend the parallelism between PDE and operator theory in important ways. Namely, under some further assumptions, we construct explicitly the eigenfunctions of the differential operator directly from the PDE characteristic matrix.
The construction does not require knowledge of the integral representation even implicitly, as neither $\Pi$ nor $\Pi'$ need be well-posed.

\bigskip

In the sequel, we assume that the boundary conditions are non-Robin and that a technical symmetry condition always holds, see Conditions~\ref{cond:non-Robin} and~\ref{cond:Symmetry} in the appendix. We also define
\BE
\zeta_j(\rho;q_0) = \sum_{r=1}^n \det\Msups{X}{}{}{r}{j}(\rho) \hat{q}_0(\omega^{r-1}\rho),
\EE
so that
\BE
\zeta^\pm(\rho;q_0) = \sum_{j=1}^n \Msup{\mathcal{A}}{1}{j}{\pm}(\rho)\zeta_j(\rho;q_0).
\EE

In the next proposition, we characterise the eigenfunctions of $S$ in terms of the PDE characteristic matrix and the spectral functions.

\begin{prop} \label{prop:non-Robin.symmetry.then.ef}
For each $k\in\N$ and for each $j\in\{1,2,\dots,n\}$, the function
\BE \label{eqn:ef.X}
\phi_k^j(x) = \sum_{r=1}^ne^{-i\omega^{r-1}\sigma_k(1-x)}\det X^{r\hspace{0.5mm}j}(\sigma_k)
\EE
is an eigenfunction of $S$ with eigenvalue $\sigma_k^n$. Further, 
\begin{align} \label{eqn:ef.zeta}
\zeta_j(\sigma_k,q_0) &= \frac{1}{C_j}\langle q_0,\psi_k^j\rangle,\qquad j=1,...,n,\quad k\in {\mathbb N} \\ \label{eqn:ef.zetastar}
\zeta_j(\bar{\sigma}_k,q_0) &= C_j\langle q_0,\phi_k^j\rangle, \\ \label{eqn:ef.adjef}
\overline{\psi_k^j(1-x)} &= C_j \phi_k^j(x),
\end{align}
where $\psi_k^j$ is the corresponding eigenfunction from the adjoint operator $S^\star$ and $C_j$ is a nonzero real scalar quantity depending only upon $j$.
\end{prop}

\begin{rmk}
The proposition above requires that the boundary conditions be non-Robin and obey the symmetry condition. These requirements may not be sharp but we have been unable to find an example failing either condition for which the result holds. 
\end{rmk}

By Proposition~\ref{prop:non-Robin.symmetry.then.ef}, the spectral functions of the original and adjoint problems, which we denote by $\zeta_j$, $\zeta_j^\star$, obey the identity
\BE \label{eqn:ef.zeta.ratio}
\frac{\zeta_j(\sigma_k,q_0)}{\zeta_j^\star(\overline{\sigma}_k, q_0)} = \frac{C_j^2\langle q_0,\psi_k^j\rangle}{\langle q_0,\phi_k^j\rangle}.
\EE

The function $q_0(x)$ denotes the initial datum of the IBVP. Hence it can be chosen arbitrarily in $\mathcal{D}(S)$. The particular choice $q_0(x)=\psi_k^j(x)$ is admissible since $\psi_k^j(x)$ is $C^\infty$ by definition. With this choice, equation~\eqref{eqn:ef.zeta.ratio} yields
\BE \label{eqn:norm.Qk}
\frac{\zeta_j(\sigma_k,\psi_k^j)}{\zeta_j^\star(\overline{\sigma}_k, \psi_k^j)} =\frac{\|\psi_k^j\|^2}{\langle \psi_k^j,\phi_k^j\rangle} = \frac{|C_j|\|\phi_k^j\|\|\psi_k^j\|}{\langle \psi_k^j,\phi_k^j\rangle} = |C_j| \|Q_k\|,
\EE
where $Q_k$ is the projection operator
\BE
Q_k(f) = \langle f, \phi_k \rangle \psi_k
\EE
considered by~\citet{Dav2007a}. Note that the latter equality follows from equation~\eqref{eqn:ef.adjef}.

By a simple change of variables we find
\BE \label{eqn:zeta.zetastar}
\zeta_j^\star(\bar{\rho},q_0(\cdot)) = -C_j\overline{\zeta_j(\rho,\bar{q}_0(1-\cdot))}.
\EE
We therefore deduce the following important result, which gives a way to control the norms of the projection operators $Q_k$ explicitly in terms of the spectral functions associated with the corresponding initial and boundary value problem. 

\begin{prop} \label{prop:Davies.Norms.1}
Let $S$ be the operator associated with $\Pi$. Then the eigenfunctions $\phi_k^j$ and $\psi_k^j$ of $S$ and of its adjoint satisfy
\BE \label{eqn:norm.Qk.2}
\frac{\|\psi_k^j\|^2}{\langle\phi_k^j,\psi_k^j\rangle} = \frac{C_j\zeta_j(\sigma_k,\psi_k^j)}{-\overline{\zeta_j(\sigma_k,\phi_k^j)}}.
\EE
\end{prop}

\begin{rmk}
This result implies that we can estimate $\|Q_k\|$ using only the spectral functions of the initial- and final-BVPs, whose construction is algorithmic.
\end{rmk}

Conversely, this proposition has an important consequence, namely an estimate on the unboundedness of the spectral functions in terms of the ``wildness'' of the  family of biorthogonal eigenfunctions of $S$. (Following~\citet{Dav2000a}, we say that a biorthogonal system is \emph{wild} if the corresponding projection operators are not uniformly bounded in norm.) We illustrate the result of this theorem in the two examples we consider in sections 3 and 4. 

\begin{thm} \label{thm:Davies.Norms.2}
Let $q_0$ be any admissible initial condition for the boundary value problem, and let $(\rho_k)_{k\in{\mathbb N}}$ be any sequence such that
\begin{itemize}
\item
$\rho_k\to\infty$ as $k\to\infty$.
\item
$|\rho_k|<|\rho_{k+1}|$
\item
$(\rho_k)$ is bounded away from the set of zeros of $\DeltaP$, uniformly in $k$:
$$\exists  \delta>0:\hsforall k,j\in\N,\quad ||\rho_k|-|\sigma_j||>\delta$$
\end{itemize}
Then
\BES
\|Q_k\|= O\left(\sup_{(\rho_k)}\left[\frac{\zeta_j(\rho_k,\psi_k^j)}{\DeltaP(\rho_k)}\cdot \frac{\DeltaP(\rho_k)} {\zeta_j(\rho_k,\phi_k^j)}\right]\right), \;\;as \;k\to\infty.
\EES
\end{thm}

\subsection{Sketch of proofs}

\begin{proof}[Proof of Theorem~\ref{thm:both.wellposed.then.complete}]
As $\Pi$ is well-posed and $\Pi'$ is well-conditioned, by~\citet{Smi2012a,Smi2013a} the solution $q$ of the problem $\Pi$ can be expressed using a series as
\BES
q(x,t) = i\sum_{k\in K^+}\res_{\rho=\sigma_k}\frac{e^{i\rho x-a\rho^nt}}{\DeltaP(\rho)} \zeta^+(\rho) + i\sum_{k\in K^-}\res_{\rho=\sigma_k}\frac{e^{i\rho(x-1)-a\rho^nt}}{\DeltaP(\rho)} \zeta^-(\rho).
\EES
As each $\sigma_k$ is a simple zero of $\DeltaP$, the series is separable into $x$-dependent and $t$-dependent parts
\begin{align}
\xi_k(x) &= \begin{cases} \frac{i}{2}e^{i\sigma_k x} \res_{\rho=\sigma_k}\frac{\zeta^+(\rho)}{\DeltaP(\rho)} & \mbox{if } k\in K^+, \\ \frac{i}{2}e^{i\sigma_k (x-1)} \res_{\rho=\sigma_k}\frac{\zeta^-(\rho)}{\DeltaP(\rho)} & \mbox{if } k\in K^-, \end{cases} \\
\tau_k(t) &= e^{-a\sigma_k^nt},
\end{align}
so that
\BE \label{eqn:both.wellposed.then.complete:q.chi.tau}
q(x,t) = \sum_{k\in\N}\xi_k(x)\tau_k(t).
\EE
Further,~\citet[Lemma~6.1]{Smi2012a} guarantees the existence of a nonzero complex constant $C$ such that $\sigma_k=Ck+O(1)$ as $k\to\infty$, which, by~\citet[Theorems~3.3.3~\&~4.1.1]{Sed2005a}, guarantees that $(\tau_k)_{k\in\N}$ is a minimal system in $L^2[0,T]$.

As $q$ is the solution of $\Pi$, $q$ satisfies
\BES
A\BP \partial_x^{n-1}q(0,t)\\\partial_x^{n-1}q(1,t)\\\vdots\\q(0,t)\\q(1,t) \EP = 0, \quad \hsforall t\in[0,T].
\EES
The minimality of the $t$-dependent system means that this implies each $\xi_k$ satisfies the boundary conditions of $S$, so $\xi_k\in\mathcal{D}(S)$.

As $q$ satisfies the PDE,
\BES
0 = a\sum_{k\in\N}[-\sigma_k^nI + S](\xi_k)(x)\tau_k(t)
\EES
so, by minimality of $(\tau_k)_{k\in\N}$, each $\xi_k$ is an eigenfunction of $S$ with eigenvalue $\sigma_k^n$.

Evaluating equation~\eqref{eqn:both.wellposed.then.complete:q.chi.tau} at $t=0$ yields an expansion of $q_0$ in the system $(\xi_k)_{k\in\N}$.
\end{proof}

\begin{rmk}
We have to require the zeros of $\DeltaP$ are all simple. It would be desirable to be able to say that the zeros of $\Delta$ and $\DeltaP$ are all the same and of the same order. It has been shown that this holds under certain symmetry restrictions on the boundary conditions~\citep{Smi2011a} and has been established in particular for all possible 3\rd order boundary conditions.
\end{rmk}

\begin{proof}[Proof of Proposition~\ref{prop:basis.then.both.wellposed}]
As $(\phi_k)_{k\in\N}$ is a basis, the Fourier expansion
\BES
q_0(x)=\sum_{k\in\N}\phi_k(x)\langle q_0,\psi_k \rangle
\EES
converges. By~\citet{Smi2013a}, well-posedness of $\Pi$ guarantees that the $\sigma_k$ are arranged in such a way that the exponential functions $e^{a\sigma_k^nt}$ are bounded uniformly in $k$, hence that the series~\eqref{eqn:basis.then.both.wellposed:q} converges for all $t\in[0,T]$. The eigenfunctions all satisfy the boundary conditions of the operator so the Fourier series satisfies the boundary conditions of the initial-value problem. The Fourier series also satisfies the partial differential equation. So we have a series representation of  the solution and $\Pi'$ must be well-conditioned.
\end{proof}

\begin{proof}[Proof of Proposition~\ref{prop:non-Robin.symmetry.then.ef}]
Let $B_l$ be the $l$\Th boundary condition of $S$. As the boundary conditions are non-Robin, they each have an order $m_l$. Hence
\BE
B_l(\phi_k^j) = \sum_{r=1}^n(i\sigma_k\omega^{r-1})^{m_l}\left[\M{\alpha}{l}{m_l}e^{-i\omega^{r-1}\sigma_k}+\M{\beta}{l}{m_l}\right]\det X^{r\hspace{0.5mm}j}(\sigma_k).
\EE
The bracketed expression is an entry from row $r$ of the characteristic matrix of $S$. Provided the boundary conditions also satisfy the symmetry condition, an algebraic manipulation yields that each column of the characteristic matrix of $S$ is a scalar multiple of a column of $\mathcal{A}$ (see the proof of~\citealp[Theorem~4.15]{Smi2011a}). So either $B_l(\phi_k^j)$ is the determinant of a matrix with a repeated column or $B_l(\phi_k^j)=\DeltaP(\sigma_k)$. In either case, $B_l(\phi_k^j)=0$, so $\phi_k^j\in\mathcal{D}(S)$. Finally, $S(\phi_k^j)=\sigma_k^n\phi_k^j$.

Let the map $r\mapsto\hat{r}$ be given by the permutation $(1,n,n-1,\ldots,3,2)$, whose sign is $(-1)^{\lceil n/2 \rceil-1}$. Because the boundary conditions obey Conditions~\ref{cond:non-Robin}--\ref{cond:Symmetry},
\BE \label{eqn:X.Xstar}
C_j \det \overline{X}^{r\hspace{0.5mm}j}(\rho) = e^{-i\omega^{\hat{r}-1}\overline{\rho}}\det X^{\hat{r}\hspace{0.5mm}j\hspace{0.5mm}\star}(\overline{\rho}), \quad \hsforall \rho\in\C
\EE
where the real constant
\BES
\frac{1}{C_j} = (-1)^{\lceil n/2 \rceil-1}\prod_{l\neq j}\beta_l,
\EES
and $\beta_l$ is the coupling constant appearing in the $l$\Th column of $\mathcal{A}$ ($1$ if there is no coupling constant in that column).

Indeed, as $S$ is a closed operator, densely-defined on $L^2[0,1]$, the eigenvalues of $S^\star$ are the points $\overline{\sigma}_k^n$, and $\overline{\sigma}_k$ are the zeros of the adjoint PDE characteristic matrix~\citep[Theorem~4.15]{Smi2011a}. Note also that the construction of the adjoint boundary conditions from the boundary conditions of the original problem~\citep[Theorem~3.2.4]{CL1955a} ensures that the adjoint boundary conditions also satisfy Conditions~\ref{cond:non-Robin}--\ref{cond:Symmetry}.

As the boundary conditions are non-Robin, the only columns that may appear in $\mathcal{A}$ are
\BE
\BP 1 \\ \omega^l \\ \vdots \\ \omega^{(n-1)l} \EP, \qquad \BP e^{-i\rho} \\ \omega^le^{-i\omega\rho} \\ \vdots \\ \omega^{(n-1)l}e^{-i\omega^{n-1}\rho} \EP, \qquad \BP (e^{-i\rho}+\beta_m) \\ \omega^l(e^{-i\omega\rho}+\beta_m) \\ \vdots \\ \omega^{(n-1)l}(e^{-i\omega^{n-1}\rho}+\beta_m) \EP,
\EE
where $l$ may vary over $\{0,1,\ldots,n-1\}$. To each of these corresponds a unique column in $\mathcal{A}^\star$, with the same values of $l$ as each column in $\mathcal{A}$: respectively,
\BE
\BP e^{-i\rho} \\ \omega^le^{-i\omega\rho} \\ \vdots \\ \omega^{(n-1)l}e^{-i\omega^{n-1}\rho} \EP, \qquad \BP 1 \\ \omega^l \\ \vdots \\ \omega^{(n-1)l} \EP, \qquad \BP (e^{-i\rho}+1/\beta_m) \\ \omega^l(e^{-i\omega\rho}+1/\beta_m) \\ \vdots \\ \omega^{(n-1)l}(e^{-i\omega^{n-1}\rho}+1/\beta_m) \EP.
\EE
Hence, to construct $\mathcal{A}(\rho)$ from $\mathcal{A}^\star(\overline{\rho})$ we apply the following operations:
\begin{enumerate}
  \item{For all $r$, multiply the $r$\Th row by $e^{i\omega^{r-1}\overline{\rho}}$.}
  \item{For all $m$, multiply the $m$\Th column by $\beta_m$.}
  \item{Apply the permutation $r\mapsto\hat{r}$ to the row index.}
  \item{Take the complex conjugate of each entry.}
\end{enumerate}
This justifies equation~\eqref{eqn:X.Xstar}.

By equation~\eqref{eqn:ef.X}, the eigenfunctions of the adjoint operator are
\BE \label{eqn:psi.Xstar1}
\psi_k^j(x) = \sum_{r=1}^ne^{-i\omega^{r-1}\overline{\sigma}_k(1-x)}\det X^{r\hspace{0.5mm}j\hspace{0.5mm}\star}(\overline{\sigma}_k).
\EE
By the definition of $r\mapsto\hat{r}$, $\omega^{1-r}=\omega^{\hat{r}-1}$. Hence
\BE \label{eqn:psi.Xstar2}
\psi_k^j(x) = \sum_{r=1}^ne^{i(\omega^{1-r}x-\omega^{\hat{r}-1})\overline{\sigma}_k}\det X^{\hat{r}\hspace{0.5mm}j\hspace{0.5mm}\star}(\overline{\sigma}_k).
\EE
Hence, by equation~\eqref{eqn:X.Xstar},
\begin{align*}
\psi_k^j(x) &= C_j \sum_{r=1}^ne^{i\omega^{1-r}\overline{\sigma}_kx}\det \overline{X}^{r\hspace{0.5mm}j}(\sigma_k) \\
&= C_j \overline{\sum_{r=1}^ne^{-i\omega^{r-1}\sigma_kx}\det X^{r\hspace{0.5mm}j}(\sigma_k)}.
\end{align*}
Hence, by the definition of $\zeta_j$, it follows that $\zeta_j(\sigma_k) = \langle q_0 , \psi_k^j \rangle/C_j$.
\end{proof}

\section{Third order coupled and uncoupled examples} \label{sec:Eg1}

In this section we outline the analysis of a particular class of boundary value problems, depending on a real parameter $\beta$, for the third order PDE $q_t=q_{xxx}$. Namely we consider the following problem: 
\begin{align}
q_t&=q_{xxx},     & &x\in[0,1], \quad t\in[0,T],\label{3beta}\\
q(x,0)&=q_0(x),   & &x\in[0,1]\notag \\
q(0,t)&=q(1,t)=0, & &q_x(0,t)+\beta q_x(1,t)=0, \quad t\in[0,T],\quad \beta\in\R\notag
\end{align}
where  $q_0\in \mathcal{D}(S)$ is a known function.

In the limit as the constant $\beta\to 0$, the second boundary condition at $x=0$ is $q_x(0,t)=0$. The spectral properties of this limiting case are very different from the case $\beta\neq 0$, when  the coupling between the first order derivatives is lost. Hence we refer to the boundary conditions corresponding to the value $\beta=0$ as  {\em uncoupled}.

In this section we analyse the behaviour of the associated differential operator in the two cases. To avoid technicalities, and to concentrate on the $\beta=0$ limit, we assume in what follows that $\beta \in(-1,1)$.

\subsection*{The associated differential operator}
Let $S^\beta$  be the differential operator corresponding to the boundary value problem (\ref{3beta}), hence specified by $n=3$ and 
by the boundary coefficient matrix 
\BE
A^\beta=\BP0&0&1&\beta&0&0\\0&0&0&0&1&0\\0&0&0&0&0&1\EP, \quad \beta\in(-1,1).
\EE

Setting
$$
\omega = e^{\frac{2\pi i}{3}} = -\frac{1}{2}+\frac{\sqrt{3}}{2},
$$
we find that the characteristic determinant~\eqref{chardet} is given by
\begin{align} \notag
\Delta^\beta(\rho) &= i\rho \sum_{j=0}^2{\omega^j(e^{-i\omega^j\rho}+\beta)(e^{-i\omega^{j+1}\rho}-e^{-i\omega^{j+2}\rho})} \\
&= i\rho (\omega-\omega^2)\left[\sum_{r=0}^2\omega^re^{i\omega^r\rho}- \beta\sum_{r=0}^2\omega^re^{-i\omega^r\rho}\right] \\ 
\mbox{ in particular } \Delta^0(\rho) &= i\rho (\omega-\omega^2)\sum_{r=0}^2\omega^re^{i\omega^r\rho}. \label{eqn:Eg:1:Setup:Uncoupled:Delta}
\end{align}

In all these cases, the PDE discrete spectrum is equal to the discrete spectrum of the operator~\citep{Smi2011a}.
 
A calculation of the associated polynomials shows that the differential operator $S^\beta$ is Birkhoff regular if $\beta\neq 0$. On the other hand, the differential operator $S^0$ obtained when $\beta=0$ is degenerate irregular by Locker's~\citeyearpar{Loc2008a} classification.

Although the only difference between the coupled and uncoupled operators is the first boundary condition, it is expected from the classification result that the operators have very different behaviour. This difference is reflected in the spectral behaviour of the two differential operators, as is shown in section~\ref{sec:Eg:1:Basis} below. The initial-boundary value problems also have very different properties. These are discussed in section~\ref{sec:Eg:1:IBVP}.

\subsection{The spectral theory} \label{sec:Eg:1:Basis}

In this section we use operator theoretic results to investigate whether the eigenfunctions of $S^{\beta}$  form a basis. 

\smallskip
{\bf The case $\beta\neq 0$.} It is shown in~\citet{Smi2011a} that this differential operator is regular, hence by the theory of~\citet{Loc2000a} we conclude that the eigenfunctions form a complete system in $\overline{\mathcal{D}}(S)$.

\smallskip
{\bf The case $\beta=0$.} Since this differential operator is degenerate irregular, Locker's theory does not apply. Indeed, the proof of the following result can be found in~\citet{Smi2011a} and also in~\citet{Pap2011a}.

\begin{thm} \label{thm:Eg:1:Basis:Uncoupled:NotBasis}
Let $S^{0}$ be the differential operator corresponding to $\beta=0$. 
Then the eigenfunctions of $S^{0}$ do not form a basis in $\overline{\mathcal{D}}(S)$.
\end{thm}

The proof is based on the following steps:
\begin{itemize}
\item
The eigenvalues of $S^{0}$ are the cubes of the nonzero zeros of the exponential polynomial
\BE \label{eqn:Eg:1:Basis:Uncoupled:Lem.Eigenvalues:NewDelta}
e^{i\rho}+\omega e^{i\omega\rho} + \omega^2 e^{i\omega^2\rho}.
\EE
The nonzero zeros of expression~\eqref{eqn:Eg:1:Basis:Uncoupled:Lem.Eigenvalues:NewDelta} may be expressed as complex numbers $\sigma_k,$ $\omega\sigma_k,$ $\omega^2\sigma_k$ for each $k\in\mathbb{N}$, where $\Re(\sigma_k)=0$ and $\Im(\sigma_k)>0$. Then $\sigma_k$ is given asymptotically by
\BE \label{eqn:Eg:1:Basis:Uncoupled:Lem.Eigenvalues:Asymptotic.Eigenvalues}
-i\sigma_k = \frac{2\pi}{\sqrt{3}}\left(k+\frac{1}{6}\right) + O\left( e^{-\sqrt{3}\pi k} \right) \mbox{ as } k\to\infty.
\EE
\item
For each $k\in\mathbb{N}$, $\phi_k$ is an eigenfunction of $S^{0}$ with eigenvalue $\sigma_k^3$, where
\BE \label{eqn:Eg:1:Basis:Uncoupled:Lem.Eigenfunctions:Eigenfunctions}
\phi_k(x) = \sum_{r=0}^2e^{i\omega^r\sigma_kx}\left(e^{i\omega^{r+2}\sigma_k}-e^{i\omega^{r+1}\sigma_k}\right),\quad k\in\mathbb{N}.
\EE
\item
The adjoint operator $(S^{0})^\star$ has eigenvalues  $-\sigma_k^3$, $k\in\mathbb{N}$, and eigenfunctions
\BE \label{lem:Eg:1:Basis:Uncoupled:Adjoint:Eigenfunctions}
\psi_k(x) = \sum_{r=0}^2e^{-i\omega^r\sigma_kx}\left(e^{-i\omega^{r+2}\sigma_k}-e^{-i\omega^{r+1}\sigma_k}\right),\quad k\in\mathbb{N}.
\EE
\item
Define $\Psi_k(x) = {\psi_k(x)}/{\langle\psi_k,\phi_k\rangle}$. Then there exists some minimal $Y\in\mathbb{N}$ such that $((\phi_k)_{k=Y}^\infty,(\Psi_k)_{k=Y}^\infty)$ is a biorthogonal sequence in $AC^n[0,1]$. Moreover
\begin{equation} \label{eqn:Eg:1:Basis:Uncoupled:Biorth.Lem:Norm}
\langle\psi_k,\phi_k\rangle 
= (-1)^k\frac{\sqrt{3}}{2}e^{\sqrt{3}\pi\left(k+\frac{1}{6}\right)} + O(1) \mbox{ as } k\to\infty.
\end{equation}
\item
The eigenfunctions have the same norm and it grows \emph{at a greater rate} than their inner product.
\begin{equation} \label{eqn:Eg:1:Basis:Uncoupled:Eigenfunction.Norm.Lem:Equal.Norms}
\|\psi_k\|^2 = \|\phi_k\|^2 
= \frac{3\sqrt{3}e^{\frac{4\pi}{\sqrt{3}}\left(k+\frac{1}{6}\right)}}{4\pi\left(k+\frac{1}{6}\right)} + O\left(\frac{e^{\frac{2\pi}{\sqrt{3}}k}}{k}\right) \mbox{ as } k\to\infty.
\end{equation}
\item
Assume $Y=1$ (if $Y>1$ the biorthogonal sequence $((\phi_k)_{k=Y}^\infty,(\psi_k)_{k=Y}^\infty)$ is not complete). Then the projections $Q_k= \|\phi_k\|\|\Psi_k\| $ are well defined, and
\begin{align} \notag
\|Q_k\| 
&= \frac{\|\phi_k\|^2}{|\langle \psi_k,\phi_k \rangle|} \\ \label{eqn:Eg:1:Basis:Uncoupled:NotBasis:Qk.Assympt}
&= \frac{3e^{\frac{\pi}{\sqrt{3}}\left(k+\frac{1}{6}\right)}}{2\pi\left( k+\frac{1}{6} \right)} + O\left( \frac{e^{-\frac{\pi}{\sqrt{3}}k}}{k} \right) \mbox{ as } k\to\infty.
\end{align}
Hence the biorthogonal sequence is wild. Now the results of~\citet[Chapter~3]{Dav2007a} show that $(\phi_k)_{k\in\mathbb{N}}$ is not a basis in $AC^n[0,1]$.
\end{itemize}

\subsubsection*{The case $\beta =0$ as a limit}

We now consider the uncoupled case as the limit $\beta\to0$ of such calculations for the coupled operator. The zeros of $\DeltaP^{\beta}$ are given by
\BE \label{eqn:Eg:1:Basis:Comp:Coupled.Lim:sigmak}
\sigma_k = \begin{cases} \left(k-\tfrac{1}{3}\right)\pi + i\log(-\beta) + O\left(e^{\frac{-\sqrt{3}k\pi}{2}}\right) & k\mbox{ even,} \\ \left(-k-\tfrac{2}{3}\right)\pi + i\log(-\beta) + O\left(e^{\frac{-\sqrt{3}k\pi}{2}}\right) & k\mbox{ odd,} \end{cases}
\EE
and the eigenfunctions of $S$ and $S^\star$ are given by equations~\eqref{eqn:Eg:1:Basis:Uncoupled:Lem.Eigenfunctions:Eigenfunctions} and~\eqref{lem:Eg:1:Basis:Uncoupled:Adjoint:Eigenfunctions} respectively using the new $\sigma_k$. After a suitable scaling, the eigenfunctions of the operator and its adjoint form a biorthogonal sequence.

A direct computation shows that for $\beta\neq0$, the fastest-growing terms in $\|\phi^{\beta}_k\|$ cancel out so that, for large $k$,
\BES
\|\phi^{\beta}_k\|^2 = O\left(e^{\frac{\sqrt{3}k\pi}{2}}k^{-1}\right), \qquad \langle \phi^{\beta}_k,\psi^{\beta}_k \rangle = O\left(e^{\frac{\sqrt{3}k\pi}{2}}k^{-1}\right).
\EES
(This cancellation does not occur in the case $\beta=0$.) This causes the projection operators $Q_n$ to be
uniformly bounded in terms of the parameter $\beta$ for $\beta\in[-1+\varepsilon,-\varepsilon]$ for every $0<\varepsilon<1$. However, the bound is not uniform as $\varepsilon\to 0$. This lack of uniformity is reflected in the transition from a regular to a degenerate irregular problem.

\smallskip
It is useful to compare the positioning of the eigenvalues, $\sigma_k^3$.  Asymptotic estimates as well as numerical evidence suggest that for any particular $\beta\in(-1,0)$ the zeros of $\DeltaP$ are distributed approximately at the crosses in Figure~\ref{fig:Eg:1:Basis:Comp:asmpbeta}; the solid rays and line segments represent the asymptotic locations of the zeros; the dashed lines are $\partial D$, the contours of integration in the associated initial-boundary value problem. As $\beta\to0^-$, hence $\log(-\beta)\to-\infty$, the solid rays move further from the origin, leaving the complex plane entirely in the limit, so that the solid line segments emanating from the origin extend to infinity.

\begin{figure}
\begin{center}
\includegraphics{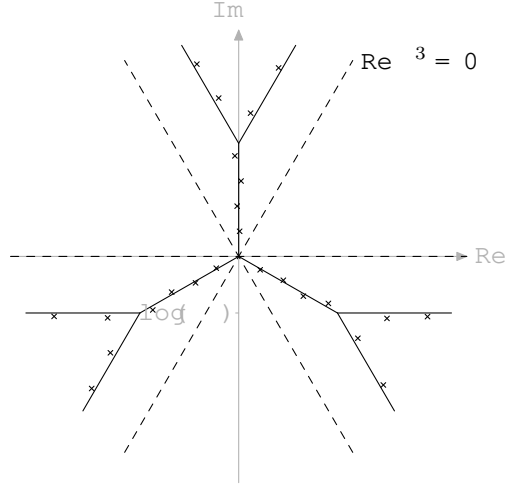}
\caption{The asymptotic position of $\sigma_k$ for $\beta\in(-1,0)$.}
\label{fig:Eg:1:Basis:Comp:asmpbeta}
\end{center}
\end{figure}

\subsection{The PDE theory} \label{sec:Eg:1:IBVP}

We now show, using the Fokas method, that while the initial-boundary value problems is well-posed for any value of $\beta$, the solution admits a series representation only if $\beta\neq 0$, in agreement with the operator theory result of the previous section.

\smallskip
\noindent {\bf The case $\beta\neq 0$}

It is already well known~\citep{FP2005a,Smi2011a} that in this case we have the following result.

\begin{thm} \label{thm:Eg:1:IBVP:Coupled:Series}
The initial-boundary value problem associated with $(S^\beta,i)$ is well-posed and its solution admits a series representation.
\end{thm}

\smallskip
\noindent {\bf The case $\beta= 0$}

\begin{thm} \label{thm:Eg:1:IBVP:Uncoupled:No.Series}
The initial-boundary value problem associated with $(S^0,i)$ is well-posed but the problem $(S^0,-i)$ is ill-conditioned.
\end{thm}

\begin{proof}
The proof of the well-posedness claim in this statement can be found in~\citet{Smi2011a}. However, for this example we now show that the statement
`$\zeta^\pm(\rho)/\DeltaP(\rho) \to0$ as $\rho\to\infty$ from within the sets enclosed by $\Gamma^\pm$'
does not hold, implying that $(S^0,-i)$ is ill-conditioned.

The reduced global relation matrix in this case is given by
\BES
\mathcal{A}(\rho) = \BP c_2(\rho) & c_2(\rho)e^{-i\rho} & c_1(\rho)e^{-i\rho} \\ c_2(\rho) & c_2(\rho)e^{-i\omega\rho} & c_1(\rho)\omega e^{-i\omega\rho} \\ c_2(\rho) & c_2(\rho) e^{-i\omega^2\rho} & c_1(\rho)\omega^2e^{-i\omega^2\rho} \EP,
\EES
hence its determinant $\DeltaP(\rho) =\Delta^{0}(\rho)$ given by (\ref{eqn:Eg:1:Setup:Uncoupled:Delta}), and the functions
\begin{align*}
\zeta_1(\rho) &= i\rho (\omega^2-\omega)\sum_{r=0}^2\omega^r\hat{q}_0(\omega^r\rho)e^{i\omega^r\rho}, \\
\zeta_2(\rho) &= i\rho \sum_{r=0}^2\hat{q}_0(\omega^r\rho)\left(\omega^{r+1}e^{-i\omega^{r+1}\rho} - \omega^{r+2}e^{-i\omega^{r+2}\rho}\right), \\
\zeta_3(\rho) &= i\rho \sum_{r=0}^2\hat{q}_0(\omega^r\rho)\left(e^{-i\omega^{r+2}\rho} - e^{-i\omega^{r+1}\rho}\right), \\
\zeta_4(\rho) &= \zeta_5(\rho) = \zeta_6(\rho) = 0.
\end{align*}
As $a=i$, the regions of interest are
\begin{align*}
\widetilde{E}_j &= E_j \setminus \{ \mbox{neighbourhoods of each } \sigma_k \}, \\
E_j &= \left\{\rho\in\mathbb{C}:\frac{(2j-1)\pi}{3}<\arg(\rho)<\frac{2j\pi}{3}\right\}.
\end{align*}

We consider the particular ratio
\BE \label{eqn:Eg:1:IBVP:Uncoupled:No.Series.Thm:Beta3.DeltaP.1}
\frac{\zeta_3(\rho)}{\DeltaP(\rho)},\quad \rho\in\widetilde{E}_2.
\EE
For $\rho\in\widetilde{E}_2$, $\Re(i\omega^r\rho)<0$ if and only if $r=2$ so we approximate ratio~\eqref{eqn:Eg:1:IBVP:Uncoupled:No.Series.Thm:Beta3.DeltaP.1} by its dominant terms as $\rho\to\infty$ from within $\widetilde{E}_2$,
\BES
\frac{(\hat{q}_0(\rho) - \hat{q}_0(\omega\rho))e^{-i\omega^2\rho} + \hat{q}_0(\omega^2\rho)(e^{-i\omega\rho} - e^{-i\rho}) + o(1)}{(\omega^2-\omega)e^{i\rho}+(1-\omega^2)e^{i\omega\rho} + o(1)}.
\EES
We expand the integrals from $\hat{q}_0$ in the numerator and multiply the numerator and denominator by $e^{-i\omega\rho}$ to obtain
\BE \label{eqn:Eg:1:IBVP:Uncoupled:No.Series.Thm:Beta3.DeltaP.2}
\frac{i\int_0^1{\left( e^{i\rho(1-x)} - e^{i\rho(1-\omega x)} - e^{i\rho\omega^2(1-x)} + e^{-i\rho(2\omega-\omega^2x)} \right)\hat{q}_0(x)}\d x + o\left(e^{\Im(\omega\rho)}\right)}{\sqrt{3}(e^{i\rho(1-\omega)}+\omega) + o\left(e^{\Im(\omega\rho)}\right)}.
\EE

Let $(R_j)_{j\in\mathbb{N}}$ be a strictly increasing sequence of positive real numbers such that $\rho_j=R_je^{i\frac{7\pi}{6}}\in\widetilde{E}_2$, $R_j$ is bounded (uniformly in $j$ and $k$) away from $\{\frac{2\pi}{\sqrt{3}}(k+\frac{1}{6}):k\in\mathbb{N}\}$
and $R_j\to\infty$ as $j\to\infty$. Then $\rho_j\to\infty$ from within $\widetilde{E}_2$. We evaluate ratio~\eqref{eqn:Eg:1:IBVP:Uncoupled:No.Series.Thm:Beta3.DeltaP.2} at $\rho=\rho_j$,
\BE \label{eqn:Eg:1:IBVP:Uncoupled:No.Series.Thm:Beta3.DeltaP.3}
\frac{i\int_0^1{\left( 2ie^{\frac{R_j}{2}(1-x)-\frac{\sqrt{3}R_j}{2}i}\sin\left( \frac{\sqrt{3}R_jx}{2} \right) - e^{-R_j(1-x)}\left( 1-e^{-\sqrt{3}R_ji} \right) \right)\hat{q}_0(x)}\d x + o\left(e^{-\frac{R_j}{2}}\right)}{\sqrt{3}(e^{-\sqrt{3}R_ji}+\omega) + o\left(e^{-\frac{R_j}{2}}\right)}.
\EE
The denominator of ratio~\eqref{eqn:Eg:1:IBVP:Uncoupled:No.Series.Thm:Beta3.DeltaP.3} is bounded away from $0$ by the definition of $R_j$ and the numerator tends to $\infty$ for any nonzero initial datum.
\end{proof}

\begin{rmk} \label{rmk:Eg:1:IBVP:Uncoupled:No.Series}
In the proof of Theorem~\ref{thm:Eg:1:IBVP:Uncoupled:No.Series} we use the example of the ratio $\frac{\zeta_3(\rho)}{\DeltaP(\rho)}$ being unbounded as $\rho\to\infty$ from within $\widetilde{E}_2$. It may be shown using the same argument that $\frac{\zeta_2(\rho)}{\DeltaP(\rho)}$ is unbounded in the same region and that both these ratios are unbounded for $\rho\in\widetilde{E}_3$ using $\rho_j=R_je^{i\frac{11\pi}{6}}$ for appropriate choice of $(R_j)_{j\in\mathbb{N}}$. However the ratio
\BES
\frac{\zeta_1(\rho)}{\DeltaP(\rho)} = \frac{\zeta^+(\rho)}{\DeltaP(\rho)}
\EES
is bounded in $\widetilde{E}_1=\widetilde{E}^+$ hence it is possible to deform the contours of integration in the upper half-plane. This permits a partial series representation of the solution to the initial-boundary value problem.
\end{rmk}

\begin{rmk} \label{rmk:Coupled.Final.Time}
For all $\beta\in(-1,1)$ the final time boundary value problem is ill-posed. The asymptotic location of the zeros of $\DeltaP$, along rays wholly contained within $\{\rho\in\C:\Re(-i\rho^3)<0\}$ means that for nozero initial data the solution exhibits instantaneous blow-up. Nevertheless, for all $\beta\in(-1,0)\cup(0,1)$ the final time problem is well-conditioned. In the case $\beta=0$, the final-time problem becomes ill-conditioned and $S$ becomes degenerate irregular under Locker's classification.

When $\beta=\pm 1$, $S$ is self-adjoint and the initial- and final-boundary value problems are both well-posed. For $|\beta|>1$, the final-boundary value problem remains well-posed but the initial-boundary value problem becomes ill-posed.
Thus the self-adjoint cases represent the transitions between well-posedness of the initial- and final-boundary value problems.
Analogous to the $\beta=0$ case, in the limit $\beta=\infty$, the initial-boundary value problem becomes ill-conditioned, the solution to the final-boundary value problem may not be represented as a series and $S$ becomes degenerate irregular.
\end{rmk}

\subsection{Comparison} \label{ssec:Eg:1:IBVP:Comp}

The explicit computation of the operator norms in section \ref{sec:Eg:1:IBVP} requires the evaluation of the biorthogonal family of eigenfunctions and the precise asymptotics for the corresponding eigenvalues.

On the other hand, the integral representation of the solution of the boundary value problem can be constructed algorithmically from the given data, without the need for any precise asymptotic information about the eigenvalues, except their asymptotic location (always along a ray for odd-order problems;~\citealp[see][Theorem~6.3]{Smi2012a}). This is sufficient for a direct analysis of the terms that blow up and prevent deformation of the contour of integration and a residue computation around the eigenvalues, thereby precluding a series representation of the solution.

In the example above, the particular term in the integral representation exhibiting this blow-up is the term
\BE
\int_0^1 2i\sin\left(\frac{\sqrt{3}R_jx}{2}\right)e^{\frac{R_j}{2}(1-x)-\frac{\sqrt{3}R_j}{2}i}dx \sim \frac 2 {R_j}e^{\frac{R_j}{2}} \ldots,
\label{3bu}
\EE
where the right hand side is obtained by an integration by parts. Note that in particular we can choose $R_j=\frac{4\pi}{\sqrt{3}}\left(k+\frac{1}{6}\right)$.

Comparing this with expression~\eqref{eqn:Eg:1:Basis:Uncoupled:Eigenfunction.Norm.Lem:Equal.Norms}, 
$$
\|\psi_k\|^2 = \|\phi_k\|^2 
= \frac{3\sqrt{3}e^{\frac{4\pi}{\sqrt{3}}\left(k+\frac{1}{6}\right)}}{4\pi\left(k+\frac{1}{6}\right)} + O\left(\frac{e^{\frac{2\pi}{\sqrt{3}}k}}{k}\right) \mbox{ as } k\to\infty.
$$
it is evident that the lack of boundedness of the norms of the operators, responsible for the lack of the properties of a basis for the eigenfunctions biorthogonal family, is exactly the same lack of boundedness in the integrand of the integral representation for the solution of the PDE, yielding a barrier to the contour deformation. Indeed, using the notation of Theorem~\ref{thm:Davies.Norms.2}, we have shown that, for this example,
\begin{gather*}
\|Q_k\| = O\left( \sup_{1\gg|\epsilon|>0}\frac{\zeta_j(\omega\sigma_k+\epsilon,\psi_k)}{\DeltaP(\omega\sigma_k+\epsilon)} \right) \mbox{and} \\
\sup_{1\gg|\epsilon|>0}\frac{\zeta_j(\omega\sigma_k+\epsilon,\psi_k)}{\DeltaP(\omega\sigma_k+\epsilon)} = O\left( \|Q_k\| \right).
\end{gather*}
This is a tighter bound on the blowup of $\|Q_k\|$ than that obtained in section~\ref{sec:Theorems}. No examples have been found that violate the tighter bound but an example is presented in section~\ref{sec:Eg2} for which $\|Q_k\|=O(1)$ while the spectral ratio grows exponentially with $k$.

\section{3\rd order pseudoperiodic examples} \label{sec:Eg2}

In this section we outline the analysis of another class of boundary value problems, depending on a real parameter $\beta$, for the linearized Korteweg-de Vries equation. Namely: 
\begin{align}
q_t&=-q_{xxx},   & &x\in[0,1], \quad t\in[0,T],\label{3betapseudo}\\
q(x,0)&=q_0(x), & &x\in[0,1],\notag \\
q(0,t)&=q(1,t), & &q_x(0,t)=-\beta q_x(1,t), \quad q_{xx}(0,t)=q_{xx}(1,t), \quad t\in[0,T] \quad \beta\in\R\notag
\end{align}
where  $q_0\in \mathcal{D}(S)$ is a known function.

For all $\beta\neq0$, these are pseudoperiodic boundary conditions. In the limit as the constant $\beta\to 2$, the boundary conditions fall into the special class of pseudoperiodic conditions for which the solution cannot be represented as a discrete series~\citep[Section~5]{Smi2012a}. As in Section~\ref{sec:Eg1}, the spectral properties of this limiting case are very different from the case $\beta\neq 2$.

In this section we analyse the behaviour of the associated differential operator in the two cases. To avoid technicalities, and to concentrate on the $\beta=2$ limit, we assume in what follows that $\beta \in(2-\epsilon,2]$.

\subsection*{The associated differential operator}

For the real parameter $\beta\in(2-\epsilon,2]$, we investigate the differential operator $S^{\beta}$ with pseudoperiodic boundary coefficient matrix
\BES
A=\BP1&-1&0&0&0&0\\0&0&1&\beta&0&0\\0&0&0&0&1&-1\EP,
\EES
and the associated initial- and final-boundary value problems $\Pi^{\beta}$ and $\Pi'^{\beta}$.

\begin{rmk}
The restriction from $\beta\in\R\setminus\{-1,0,1/2\}$ to $\beta\in(2-\epsilon,2]$ is not of any consequence other than notational convenience but the cases $\beta=-1$, $\beta=0$ and $\beta=1/2$ require special treatment.

Indeed, $\beta=1/2$ is equivalent to the final-boundary value problem $\Pi'$ being well-posed but with solution lacking a series representation (as $\Pi$ is ill-conditioned) and, as $S^{1/\beta}$ is the adjoint of $S^{\beta}$ for any $\beta\neq0$, the below analysis carries over to this case with a relabeling between $S$ and $S^\star$. 

If $\beta=0$ then the boundary conditions are no longer pseudo-periodic. A description of well-posedness for this case is given in~\citet{Smi2013a}.

If $\beta=-1$ then the operator is periodic hence, from the classical theory, its eigenfunctions form a basis and the problems $\Pi$ and $\Pi'$ are both well-posed.
\end{rmk}

\subsection{The spectral theory} \label{ssec:Pseudo.Spectral}

In this section we attempt to use operator theoretic results to investigate whether the eigenfunctions of $S^{\beta}$ form a basis. 

\smallskip
\noindent {\bf The case $\beta<2$}

It is shown by~\citet{Smi2011a} that this differential operator is regular, hence by the theory of~\citet{Loc2000a} we conclude that the eigenfunctions form a complete system in $\overline{\mathcal{D}}(S)$.

\smallskip
\noindent {\bf The case $\beta=2$}

Since this differential operator is degenerate irregular, Locker's theory does not apply. However, in this example we are \emph{unable} to apply Davies' method to discern whether the eigenfunctions form a basis. The eigenfunctions form a \emph{tame}~\citep[in the sense of][]{Dav2000a} system, which is a necessary but not sufficient condition for a basis.

Indeed, following the same outline method as in Section~\ref{sec:Eg:1:Basis}, we obtain

\begin{itemize}
\item
The eigenvalues of $S^{2}$ are the cubes of the nonzero zeros of the exponential polynomial
\BE \label{eqn:Eg:2:Basis:Uncoupled:Lem.Eigenvalues:NewDelta}
e^{-i\rho} + e^{-i\omega\rho} +  e^{-i\omega^2\rho} - 3.
\EE
The nonzero zeros of expression~\eqref{eqn:Eg:2:Basis:Uncoupled:Lem.Eigenvalues:NewDelta} may be expressed as complex numbers $\sigma_k,$ $\omega\sigma_k,$ $\omega^2\sigma_k$ for each $k\in\mathbb{N}$, where $\Re(\sigma_k)=0$ and $\Im(\sigma_k)<0$. Then $\sigma_k$ is given asymptotically by
\BE \label{eqn:Eg:2:Basis:Uncoupled:Lem.Eigenvalues:Asymptotic.Eigenvalues}
i\sigma_k = \frac{2\pi}{\sqrt{3}}\left(k-\frac{1}{2}\right) + O\left( e^{-k\pi\sqrt{3}/3} \right) \mbox{ as } k\to\infty.
\EE
\item
Let
\BE \label{eqn:Eg:2:Basis:Uncoupled:Lem.Eigenfunctions:Eigenfunctions}
\phi_k(x) = \sum_{r=0}^2\omega^{r}e^{i\omega^r\sigma_kx}\left(e^{-i\omega^{r}\sigma_k}-e^{i\omega^{r+2}\sigma_k}-e^{i\omega^{r+1}\sigma_k}+1\right),\quad k\in\mathbb{N}.
\EE
Then, for each $k\in\mathbb{N}$, $\phi_k$ is an eigenfunction of $S^{2}$ with eigenvalue $\sigma_k^3$.
\item
The adjoint operator $(S^{2})^\star$ has  eigenvalues  $\{-\sigma_k^3:k\in\mathbb{N}\}$, corresponding to eigenfunctions
\BE \label{lem:Eg:2:Basis:Uncoupled:Adjoint:Eigenfunctions}
\psi_k(x) = \sum_{r=0}^2\omega^{r}e^{-i\omega^r\sigma_kx}\left(e^{i\omega^{r}\sigma_k}-e^{-i\omega^{r+2}\sigma_k}-e^{-i\omega^{r+1}\sigma_k}+1\right),\quad k\in\mathbb{N}.
\EE
and there are at most finitely many eigenfunctions of $(S^{(2)})^\star$ that are not in the set $\{\psi_k:k\in\mathbb{N}\}$.
\item
Let
\BE
\Psi_k(x) = \frac{\psi_k(x)}{\langle\psi_k,\phi_k\rangle}.
\EE
Then there exists a minimal $Y\in\mathbb{N}$ such that $((\phi_k)_{k=Y}^\infty,(\Psi_k)_{k=Y}^\infty)$ is a biorthogonal sequence in $AC^n[0,1]$. Moreover
\begin{equation} \label{eqn:Eg:2:Basis:Uncoupled:Biorth.Lem:Norm}
\langle\psi_k,\phi_k\rangle 
= \frac{\sqrt{3}}{\left(k-\frac{1}{2}\right)\pi}e^{\frac{4\pi}{\sqrt{3}}\left(k-\frac{1}{2}\right)} + O(e^{\sqrt{3}\pi k}k^{-1}) \mbox{ as } k\to\infty.
\end{equation}
\item
The eigenfunctions have the same norm and it grows \emph{at the same rate} as their inner product.
\begin{equation} \label{eqn:Eg:2:Basis:Uncoupled:Eigenfunction.Norm.Lem:Equal.Norms}
\|\psi_k\|^2 = \|\phi_k\|^2 
= \frac{3\sqrt{3}}{2\left(k-\frac{1}{2}\right)\pi}e^{\frac{4\pi}{\sqrt{3}}\left(k-\frac{1}{2}\right)} + O(e^{\sqrt{3}\pi k}k^{-1}) \mbox{ as } k\to\infty.
\end{equation}
\item
Then the projection $Q_k$ has norm $\|\phi_k\|\|\Psi_k\|$, which is bounded uniformly in $k$. From this result, it is impossible to determine whether the eigenfunctons form a basis or not.
\end{itemize}

\subsection{The PDE theory}
As shown in~\citet[Example~5.2]{Smi2012a}, $\Pi$ is ill-posed if and only if $\beta=2$. Via Proposition~\ref{prop:basis.then.both.wellposed}, this yields the result that the analysis of section~\ref{ssec:Pseudo.Spectral} could not---the eigenfunctions do not form a basis.

\begin{prop} \label{prop:Pseudo.2.Blowup}
Let $R_k=4k\pi/\sqrt{3}$ and let $\rho_k=R_ke^{i\pi/6}$. Then, using the notation of~\citet{Smi2012a}, the ratio
\BES
\frac{\eta_2^{(2)}(\rho_k)}{\DeltaPsup{(2)}(\rho_k)} = \frac{(-1)^k(q_T(0)-2q_T(1))e^{R_k/2}}{6R_k^2} + O(e^{R_k/2}R_k^{-3}), \mbox{ as } k\to\infty.
\EES
\end{prop}
\begin{proof}
A quick calculation yields
\BE
\DeltaPsup{(\beta)}(\rho_k) = i\sqrt{3}R_k^3 \left[ 3\beta+3 + (\beta-2)\sum_{j=0}^2e^{i\omega^j\rho_k} + (1-2\beta)\sum_{j=0}^2e^{-i\omega^j\rho_k} \right],
\EE
hence
\BE
\DeltaPsup{(2)}(\rho_k) = i3\sqrt{3}R_k^3 \left[ 3 - \sum_{j=0}^2e^{-i\omega^j\rho_k} \right].
\EE
The spectral function
\BES
\eta_2^{(\beta)}(\rho_k) = i\sqrt{3}\omega^2R_k^2 \sum_{j=0}^2\omega^{2j}\hat{q}_T(\omega^j\rho_k) \left( e^{i\omega^j\rho_k} - e^{-i\omega^{j+1}\rho_k} - e^{-i\omega^{j+2}\rho_k} + 1 \right)
\EES
is independent of $\beta$.

By the definition of $\rho_k$, the functions
\BES
e^{i\omega^2\rho_k}=e^{R_k},\quad e^{-i\rho_k}=e^{R_k/2}e^{iR_k\sqrt{3}/2} \quad \mbox{and} \quad e^{-i\omega\rho_k}=e^{R_k/2}e^{-iR_k\sqrt{3}/2}
\EES
grow exponentially with $k$, while
\BES
e^{-i\omega^2\rho_k}=e^{-R_k},\quad e^{i\rho_k}=e^{-R_k/2}e^{-iR_k\sqrt{3}/2} \quad \mbox{and} \quad e^{i\omega\rho_k}=e^{-R_k/2}e^{iR_k\sqrt{3}/2}
\EES
decay. Hence
\BES
\DeltaPsup{(2)}(\rho_k) = (-1)^{k+1}i6\sqrt{3}R_k^3 e^{R_k/2} + O(e^{-R_k}R_k^3), \mbox{ as } k\to\infty.
\EES
Also
\begin{align*}
\eta_2^{(2)}(\rho_k) &= -i\omega^2\sqrt{3}R_k^2 \left( - q_T(\rho_k)e^{-i\omega\rho_k} - \omega^2q_T(\omega\rho_k)e^{-i\rho_k} + \omega q_T(\omega^2\rho_k)e^{i\omega^2\rho_k} \right) \\ &\hspace{90mm}+ O(e^{R_k/2}R_k^2) \\
&= -\omega^2i\sqrt{3}R_k^2 \left( - \int_0^1q_T(x)e^{\frac{R_k}{2}[x+1 + i\sqrt{3}(x-1)]}\d x \right. \\ &\hspace{18mm} \left.- \omega^2\int_0^1q_T(x)e^{\frac{R_k}{2}[x+1 - i\sqrt{3}(x-1)]}\d x+ \omega\int_0^1q_T(x)e^{R_k[1-x]}\d x \right) \\ &\hspace{90mm} + O(e^{R_k/2}R_k^2) \\
&= -\omega^2i\sqrt{3}R_k^2 \left( - \frac{2q_T(1)e^{R_k}}{R_k(1+i\sqrt{3})} - \omega^2\frac{2q_T(1)e^{R_k}}{R_k(1-i\sqrt{3})} + \omega\frac{q_T(0)e^{R_k}}{R_k}\right.  \\ &\hspace{67mm} \left.+ O(e^{R_k}R_k^{-2}) \right) + O(e^{R_k/2}R_k^2) \\
&= -i\sqrt{3}R_k \left( 2q_T(1) - q_T(0) \right) e^{R_k} + O(e^{R_k}).
\end{align*}
Note that $q_T(1) = q(1,T) = -q(0,T) = -q_T(0)$, by the first boundary condition. Hence, provided we can be sure that $q_T(0)\neq0$, $2q_T(1) - q_T(0)\neq0$.
\end{proof}

As $0<\arg(\rho_k)<\pi/3$, and $R_k$ was chosen to ensure that $\DeltaP(\rho_k)$ is bounded away from $0$, $\rho_k\in\widetilde{D}_1$. Hence, by~\citet[Theorem~1.1]{Smi2012a}, $\Pi$ is ill-posed.

\smallskip

The rate of blowup exhibited in Proposition~\ref{prop:Pseudo.2.Blowup} is maximal in the sense that for any sequence $(\rho_k)_{k\in\N}$ such that $|\sigma_{k-1}|<|\rho_k|<|\sigma_k|$ and for any $j\in\{1,2,3\}$,
\BES
\frac{\eta_j^{(2)}(\rho_k)}{\DeltaPsup{(2)}(\rho_k)} = O(e^{R_k/2}R_k^{-2}).
\EES

The problem $\Pi'$ is well-conditioned for all $\beta\in(2-\epsilon,2]$. Indeed, for any sequence $(\rho_k)_{k\in\N}$ with $\rho_k\in \widetilde{D}_r$ and $\rho_k\to\infty$, we find the asymptotic behaviour:
\BES
\frac{\eta_j^{(\beta)}(\rho_k)}{\DeltaPsup{(\beta)}(\rho_k)} = O(|\rho_k|^{-1}).
\EES

\subsection{Comparison}
In order to find the asymptotic behaviour of $\|Q_k\|$, the complex calculation outlined in section~\ref{ssec:Pseudo.Spectral} is necessary. However, the result we obtain is that the projection operators are uniformly bounded in norm, from which we cannot discern whether the eigenfunctions form a basis.

The calculation required to prove Proposition~\ref{prop:Pseudo.2.Blowup} is relatively simple and from that result, via Proposition~\ref{prop:basis.then.both.wellposed}, it follows that the eigenfunctions are not a basis.

\section*{Conclusions}

In this paper, we have gathered and summarised old and new results on a newly analysed correspondence between the spectral theory of linear differential operators with constant coefficients and the analysis and solution of IBVPs for linear constant coefficient evolution PDEs.  
We also presented two specific examples to illustrate the power of this connection for inferring results on the spectral structure of the operator.

In Section~\ref{sec:Theorems}, we developed a new method for showing that the eigenfunctions of certain differential operators do not form a basis. This method relies crucially upon finding a well-posed IBVP whose solution cannot be represented as a series.

In Sections~\ref{sec:Eg1}--\ref{sec:Eg2}, we compare the new method to the established method of Davies by applying each method to examples. The calculations we present  suggest that the PDE approach is more straightforward in deriving estimates for the boundedness of projector operators, hence results on the existence of eigenfunction bases. Indeed, it is sufficient to estimate the boundedness of functions constructed algorithmically in certain well defined complex directions. 

The second example represents a case where the new method yields a result but the operator-theoretic methods we considered do not. Indeed we show that the solution of the only well-posed initial-boundary value problem cannot be represented as a discrete series hence, by Proposition~\ref{prop:basis.then.both.wellposed}, the eigenfunctions cannot form a basis. But the eigenfunctions are not wild, indeed the associated projection operators are uniformly bounded in norm, so we cannot reach the same conclusion using e.g.\ the operator-theoretic framework of Davies.

The remainder of Section~\ref{sec:Theorems} investigates the relation between the two methods. Indeed, for the class of operators we discuss, determining the wildness of the eigenfunctions is equivalent to the calculation of precisely the same quantities used to determine well-posedness of the associated initial-boundary value problems.

It is expected that the well-posedness of both the initial- and final-boundary value problems is sufficient to guarantee that the projection operators are uniformly bounded in norm. 

The applicability of the new method has only been shown for eigenfunctions of the class of differential operators considered herein, whereas Davies' method could be applied to any complete biorthogonal system, whether it is constructed from the eigenfunctions of differential operator or not. However, it should be possible to extend the new method, along with the results of~\citet{Smi2012a} to a wider class of differential operators, providing a powerful tool to investigate the spectral properties of linear differential operators. For example, throughout this work we have assumed that $S=(-i\partial_x)^n$. A general constant-coefficient linear differential operator may have more terms, but its principal part could always be represented by such an operator $S$. As the spectral behaviour of the operator is governed by its principal part, we expect  the above results to carry over to such operators.

\section*{Acknowledgements}
We are very grateful to
E B Davies, M Marletta and the reviewers
for their useful suggestions. 

The research leading to these results has received funding from EPSRC and the European Union's Seventh Framework Programme (FP7-REGPOT-2009-1) under grant agreement n$^\circ$ 245749.

\appendix

\section{Appendix} \label{app:Symmetry.Condition}
The statements of results~\ref{prop:non-Robin.symmetry.then.ef}--\ref{thm:Davies.Norms.2} all require the following additional conditions:

\begin{cond} \label{cond:non-Robin}
The boundary coefficient matrix $A$ is \emph{non-Robin}: \\
None of the boundary conditions represent couplings between different orders of boundary function. \\ That is, for each $k\in\{1,2,\dots,n\}$, if $\M{\alpha}{k}{j}\neq0$ or $\M{\beta}{k}{j}\neq0$ then $\M{\alpha}{k}{r}=0=\M{\beta}{k}{r}$ for all $r\neq j$.
\end{cond}
Note the following contrast with Robin's original definition. Our Robin/non-Robin classification is independent of coupling between the two ends of the interval; the boundary condition $q_x(0,t)=q(1,t)$ is of Robin type and couples the ends of the interval.

\begin{cond} \label{cond:Symmetry}
Recall that $A$ is reduced row-echelon form. The boundary conditions are such that if the boundary function of order $r$ at one end corresponds to a pivoting entry in the boundary coefficient matrix $A$ then the boundary function of order $n-1-r$ at the other end must correspond to a non-pivoting entry in $A$. Further, the coupling constants for coupled boundary conditions of order $r$ and $n-1-r$ are equal.
\end{cond}
For simple boundary conditions, this means that if the boundary function of order $r$ at one end is specified then the boundary function of order $n-1-r$ at the other end must not be specified.

\bibliographystyle{plainnat}
\bibliography{dbrefs}

\begin{thebibliography}{21}
\providecommand{\natexlab}[1]{#1}
\providecommand{\url}[1]{\texttt{#1}}
\expandafter\ifx\csname urlstyle\endcsname\relax
  \providecommand{\doi}[1]{doi: #1}\else
  \providecommand{\doi}{doi: \begingroup \urlstyle{rm}\Url}\fi

\bibitem[Birkhoff(1908{\natexlab{a}})]{Bir1908a}
G.~D. Birkhoff.
\newblock On the asymptotic character of the solutions of certain linear
  differential equations containing a parameter.
\newblock \emph{Trans. Amer. Math. Soc.}, 9:\penalty0 219--231,
  1908{\natexlab{a}}.

\bibitem[Birkhoff(1908{\natexlab{b}})]{Bir1908b}
G.~D. Birkhoff.
\newblock Boundary value and expansion problems of ordinary linear differential
  equations.
\newblock \emph{Trans. Amer. Math. Soc.}, 9:\penalty0 373--395,
  1908{\natexlab{b}}.

\bibitem[Coddington and Levinson(1955)]{CL1955a}
E.~A. Coddington and N.~Levinson.
\newblock Theory of ordinary differential equations.
\newblock In \emph{International Series in Pure and Applied Mathematics}.
  McGraw-Hill, 1955.

\bibitem[Davies(2000)]{Dav2000a}
E.~B. Davies.
\newblock Wild spectral behaviour of anharmonic oscillators.
\newblock \emph{Bull. Lond. Math. Soc}, 32:\penalty0 432--438, 2000.

\bibitem[Davies(2007)]{Dav2007a}
E.~B. Davies.
\newblock Linear operators and their spectra.
\newblock In \emph{Cambridge Studies in Advanced Mathematics}, volume 106.
  Cambridge University Press, 2007.

\bibitem[Dunford and Schwartz(1963)]{DS1963a}
N.~Dunford and J.~T. Schwartz.
\newblock Linear operators part {II} spectral theory, self adjoint operators in
  a {H}ilbert space.
\newblock In \emph{Pure and Applied Mathematics}. Wiley-Interscience, 1963.

\bibitem[Fokas(2008)]{Fok2008a}
A.~S. Fokas.
\newblock \emph{A Unified Approach to Boundary Value Problems}.
\newblock CBMS-SIAM, 2008.

\bibitem[Fokas and Pelloni(2001)]{FP2001a}
A.~S. Fokas and B.~Pelloni.
\newblock Two-point boundary value problems for linear evolution equations.
\newblock \emph{Math. Proc. Cambridge Philos. Soc.}, 131:\penalty0 521--543,
  2001.

\bibitem[Fokas and Pelloni(2005)]{FP2005a}
A.~S. Fokas and B.~Pelloni.
\newblock A transform method for linear evolution {PDE}s on a finite interval.
\newblock \emph{IMA J. Appl. Math.}, 70:\penalty0 564--587, 2005.

\bibitem[Fokas and Smith(2013)]{FS2013a}
A.~S. Fokas and D.~A. Smith.
\newblock Evolution {P}{D}{E}s and augmented eigenfunctions. {I} finite
  interval.
\newblock (submitted), 2013.

\bibitem[Fokas and Sung(1999)]{FS1999a}
A.~S. Fokas and L.~Y. Sung.
\newblock Initial-boundary value problems for linear dispersive evolution
  equations on the half-line.
\newblock (unpublished), 1999.

\bibitem[Locker(2000)]{Loc2000a}
J.~Locker.
\newblock Spectral theory of non-self-adjoint two-point differential operators.
\newblock In \emph{Mathematical Surveys and Monographs}, volume~73. American
  Mathematical Society, Providence, Rhode Island, 2000.

\bibitem[Locker(2008)]{Loc2008a}
J.~Locker.
\newblock Eigenvalues and completeness for regular and simply irregular
  two-point differential operators.
\newblock In \emph{Memoirs of the {A}merican Mathematical Society}, number 911.
  American Mathematical Society, Providence, Rhode Island, 2008.

\bibitem[Papanicolaou(2011)]{Pap2011a}
G.~Papanicolaou.
\newblock An example where separation of variable fails.
\newblock \emph{J. Math. Anal. Appl.}, 373\penalty0 (2):\penalty0 739--744,
  2011.

\bibitem[Pelloni(2004)]{Pel2004a}
B.~Pelloni.
\newblock Well-posed boundary value problems for linear evolution equations on
  a finite interval.
\newblock \emph{Math. Proc. Cambridge Philos. Soc.}, 136:\penalty0 361--382,
  2004.

\bibitem[Pelloni(2005)]{Pel2005a}
B.~Pelloni.
\newblock The spectral representation of two-point boundary-value problems for
  third-order linear evolution partial differential equations.
\newblock \emph{Proc. R. Soc. Lond. Ser. A Math. Phys. Eng. Sci.},
  461:\penalty0 2965--2984, 2005.

\bibitem[Sedletskii(2005)]{Sed2005a}
A.~M. Sedletskii.
\newblock Analytic {F}ourier transforms and exponential approximations. {I}.
\newblock \emph{J. Math. Sci. (N. Y.)}, 129\penalty0 (6):\penalty0 5083--5255,
  2005.

\bibitem[Smith(2011)]{Smi2011a}
D.~A. Smith.
\newblock \emph{Spectral theory of ordinary and partial linear differential
  operators on finite intervals}.
\newblock Phd, University of Reading, 2011.

\bibitem[Smith(2012)]{Smi2012a}
D.~A. Smith.
\newblock Well-posed two-point initial-boundary value problems with arbitrary
  boundary conditions.
\newblock \emph{Math. Proc. Cambridge Philos. Soc.}, 152:\penalty0 473--496,
  2012.

\bibitem[Smith(2013{\natexlab{a}})]{Smi2013a}
D.~A. Smith.
\newblock Well-posedness and conditioning of 3rd and higher order two-point
  initial-boundary value problems.
\newblock (submitted), 2013{\natexlab{a}}.

\bibitem[Smith(2013{\natexlab{b}})]{Smi2013b}
D.~A. Smith.
\newblock Classification of {B}irkhoff-degenerate-irregular two-point
  differential operators through associated initial-boundary value problems.
\newblock (in preparation), 2013{\natexlab{b}}.

\end{thebibliography}

\end{document}